\documentclass[12pt]{article}

\usepackage{amssymb}
\usepackage{fnpct}
\usepackage{amsmath}
\usepackage[pdftex]{graphicx}
\usepackage{epsfig}
\usepackage{framed}
\usepackage{etoolbox}
\usepackage{fnpct}
\usepackage{overpic}
\usepackage{url}
\usepackage[font=small,labelfont=bf]{caption} 
\usepackage{relsize} 
\usepackage{setspace}
\usepackage{tikz}
\usepackage{pdfpages}
\usetikzlibrary{arrows,positioning}
\usepackage{enumitem}

\usepackage[utf8]{inputenc}

\newcommand{\diag}{{\rm diag}}
\newcommand{\SNR}{{\rm SNR}}
\newcommand{\Cond}{{\rm Cond}}

\newcommand{\spa}{{\rm span \,}}

\newcommand\numberthis{\addtocounter{equation}{1}\tag{\theequation}}

\DeclareMathOperator*{\Tr}{Tr}

\DeclareMathOperator*{\Var}{Var}

\def\hide #1 {}
\long\def\longhide #1 {}

\usepackage{amsthm}

\theoremstyle{plain}
\newtheorem{theorem}{Theorem}[section]
\newtheorem*{theorem*}{Theorem}
\newtheorem{lemma}[theorem]{Lemma}
\newtheorem{proposition}[theorem]{Proposition}

\newtheorem{corollary}[theorem]{Corollary}
\newtheorem{conjecture}[theorem]{Conjecture}

\theoremstyle{definition}
\newtheorem{remark}[theorem]{Remark}
\newtheorem{example}[theorem]{Example}

\newtheorem{definition}[theorem]{Definition}

\theoremstyle{definition}

\newcommand*{\Scale}[2][4]{\scalebox{#1}{$#2$}}%

\newcounter{nootje}
\renewcommand{\check}[1]
  {\marginpar{\tiny\begin{minipage}{20mm}\begin{flushleft}\thenootje : #1\end{flushleft}\end{minipage}}\addtocounter{nootje}{1}}
\title{Extreme Singular Values of Random Time-Frequency Structured Matrices}
\author{
   Palina Salanevich\thanks{Department of Mathematics, University of California, Los Angeles, US. {\sl Email}: psalanevich@math.ucla.edu.}
}

\begin{document}
\maketitle
\date{}

\begin{abstract}
In this paper, we investigate extreme singular values of the analysis matrix of a Gabor frame $(g, \Lambda)$ with a random window $g$. Columns of such matrices are time and frequency shifts of $g$, and $\Lambda\subset \mathbb{Z}_M\times\mathbb{Z}_M$ is the set of time-frequency shift indices. Our aim is to obtain bounds on the singular values of such random time-frequency structured matrices for various choices of the frame set $\Lambda$, and to investigate their dependence on the structure of $\Lambda$, as well as on its cardinality. We also compare the results obtained for Gabor frame analysis matrices with the respective results for matrices with independent identically distributed entries.\par
\medskip
{\bf Index terms:} structured random matrices, extreme singular values, Gabor frames, time-frequency structured matrices, condition number.
\end{abstract}



\section{Introduction}

Study of the distribution of singular values plays an important role in random matrix theory. One of the motivations to study extreme values of random matrices comes from frame theory.

In the finite dimensional setup, we call a set of vectors ${\Phi = \{\varphi_j\}_{j = 1}^N\subset \mathbb{C}^M}$ a \emph{frame} with \emph{frame bounds} $0< A \leq B$ if, for any $x\in \mathbb{C}^M$, the following inequality holds
\begin{equation*}
A||x||_2^2\le \sum_{j = 1}^N |\langle x, \varphi_j\rangle|^2 \le B||x||_2^2.
\end{equation*}
\noindent In the case when frame bounds can be chosen so that $A = B$, the frame $\Phi$ is called \emph{tight}. The values $\langle x, \varphi_j\rangle$, $j\in \{1,\dots, N\}$, are called the \emph{frame coefficients} of $x$ with respect to the frame~$\Phi$. 

We note that the above inequality holds for some $0< A \leq B<\infty$ if and only if ${\spa (\Phi) = \mathbb{C}^M}$. That is, the notion of a frame is equivalent to the notion of a spanning set of~$\mathbb{C}^M$  in the finite dimensional case. In particular, we have $|\Phi| = N\ge M$. 

By a slight abuse of notation, we identify a frame $\Phi = \{\varphi_j\}_{j = 1}^N\subset \mathbb{C}^M$ with its \emph{synthesis matrix} $\Phi$, having the frame vectors $\varphi_j$ as its columns. The adjoint $\Phi^*$ of the synthesis matrix is called the \emph{analysis matrix} of the frame $\Phi$, and the product $\Phi \Phi^*$ is called the \emph{frame operator} of the frame $\Phi$. To reconstruct a vector from its frame coefficients, one can use a \emph{dual frame} $\tilde{\Phi} = \{\tilde{\varphi}_j\}_{j = 1}^N$, defined so that $x = \sum_{j=1}^N \langle x, \varphi_j\rangle \tilde{\varphi}_j$, for each $x\in \mathbb{Z}_M$.  A dual frame is not uniquely defined if $|\Phi|>M$. The \emph{standard dual frame} of $\Phi$ is given by the Moore-Penrose pseudoinverse $(\Phi \Phi^*)^{-1}\Phi$ of the synthesis matrix $\Phi$. For a complete background on frames in finite dimensions, we refer the reader to \cite{finite_frames_book}.

Frames proved to be a powerful tool in many areas of applied mathematics, computer science, and engineering. The investigation of geometric properties of frames, such as extreme singular values of their analysis matrices, plays a crucial role in many signal processing problems. Among such problems are communication systems, where the frame coefficients are used to transmit a signal over the communication channel; image processing; and also tomography, speech recognition and brain imaging, where the initial signal is not available, but we have access to its measurements in the form of the frame coefficients instead.  One of the key advantages of a frame compared to a basis is the redundancy of the signal representation using frame coefficients. Provided we have a control on the frame bounds, this redundancy allows, among other things, to achieve robust reconstruction of a signal from its frame coefficients that are corrupted by noise, rounding error due to quantization, or erasures. 

Indeed, consider a frame $\Phi = \{\varphi_j\}_{j = 1}^N\subset \mathbb{C}^M$. The optimal lower and upper frame bounds of $\Phi$ are given by 
\begin{equation*}
\begin{gathered}
A = \min_{x\in \mathbb{S}^{M-1}} \sum_{j = 1}^N|\langle x, \varphi_j \rangle|^2 = \min_{x\in \mathbb{S}^{M-1}}||\Phi \Phi^*x||_2^2 = \sigma_{\min}^2(\Phi^*),\\
B = \max_{x\in \mathbb{S}^{M-1}} \sum_{j = 1}^N|\langle x, \varphi_j \rangle|^2 =  \max_{x\in \mathbb{S}^{M-1}}||\Phi \Phi^*x||_2^2 = \sigma_{\max}^2(\Phi^*),
\end{gathered}
\end{equation*}
\noindent where $\mathbb{S}^{M-1} = \{x\in \mathbb{C}^M,  ||x||_2 = 1\}$ denotes the complex unit sphere, and  $\sigma_{\min}(A)$ and $\sigma_{\max}(A)$ denote the smallest and the largest singular values of a matrix $A$, respectively.

Let $c\in \mathbb{C}^N$ be a vector of noisy frame coefficients of a signal $x\in \mathbb{C}^M$ with respect to the frame $\Phi$. That is,
\begin{equation*}
c=\Phi^*x+ \delta,
\end{equation*}
\noindent where $\delta\in \mathbb{C}^N$ is a noise vector. Then an estimate $\tilde{x}$ of the initial signal $x$ can be obtained from its noisy measurements $c$ using the standard dual frame of $\Phi$. More precisely, we have
\begin{equation*}
\tilde{x} = (\Phi \Phi^*)^{-1}\Phi c = x + (\Phi \Phi^*)^{-1}\Phi\delta.
\end{equation*}
\noindent Thus, for the reconstruction error we have
\begin{equation*}
||\tilde{x} - x||_2^2 \le ||(\Phi \Phi^*)^{-1}\Phi ||_2^2||\delta||_2^2 = \frac{||\delta||_2^2}{\sigma_{\min}^2(\Phi^*)}.
\end{equation*}
Moreover, if we know a bound on the signal to noise ratio $\SNR = \frac{||\Phi^* x||_2}{||\delta||_2}$ for the channel used, then the norm of the reconstruction error $||(\Phi \Phi^*)^{-1}\Phi\delta||_2$ compares to the norm of the initial signal $||x||_2$ as
\begin{equation*}
\frac{||(\Phi \Phi^*)^{-1}\Phi\delta||_2}{||x||_2}\le \frac{\Cond(\Phi^*)}{\SNR}.
\end{equation*}
Here, 
\begin{equation*}
\begin{split}
\Cond(\Phi^*) & = \sup_{x\in \mathbb{C}^M\setminus \{0\}}\sup_{\delta\in \mathbb{C}^N\setminus \{0\}} \SNR\frac{||(\Phi \Phi^*)^{-1}\Phi\delta||_2}{||x||_2} \\
& = \sup_{x\in \mathbb{C}^M\setminus \{0\}} \frac{||\Phi^* x||_2}{||x||_2} \sup_{\delta\in \mathbb{C}^N\setminus \{0\}}\frac{||(\Phi \Phi^*)^{-1}\Phi\delta||_2}{||\delta||_2} = \frac{\sigma_{\max}(\Phi^*)}{\sigma_{\min}(\Phi^*)} = \frac{\sqrt{B}}{\sqrt{A}}.
\end{split}
\end{equation*}
\noindent That is, $\Cond(\Phi^*)$ is equal to the condition number of the analysis matrix of the frame~$\Phi$.

Thus, frame bounds, or extreme singular values of the frame analysis matrix, indicate the ``quality'' of a frame in the sense of the robustness of the reconstruction of an initial signal from its noisy frame coefficients. In the case when frame bounds of $\Phi$ are sufficiently close to each other, that is, when $\Cond(\Phi^*)$ is not too large, we call the frame $\Phi$ \emph{well-conditioned}.

Extreme singular values are sufficiently well-studied for random matrices with independent entries, which can be viewed as analysis matrices of randomly generated frames with independent frame vectors (see Section \ref{sec_overwiew_propert} for some results). At the same time, the concrete application for which a signal processing problem is studied usually dictates the structure of the frame used to represent a signal. This motivates the study of properties of structured random matrices corresponding to application relevant frames, such as Gabor frames.

\begin{definition}[Gabor frames]\label{Gabor_def}\mbox{}
\begin{enumerate}[leftmargin=*]
\item \emph{Translation} (or \emph{time shift}) by $k\in \mathbb{Z}_M$, is given by
\begin{equation*}
T_k x = T_k\left( x(0), x(1),\dots, x(M-1)\right) = \left(x(m-k)\right)_{m\in \mathbb{Z}_M}.
\end{equation*}
\noindent That is, $T_k$ permutes entries of $x$ using $k$ cyclic shifts.
\item \emph{Modulation} (or \emph{frequency shift}) by $\ell\in \mathbb{Z}_M$ is given by 
\begin{equation*}
M_{\ell} x = M_{\ell}\left( x(0), x(1),\dots, x(M-1)\right) = \left(e^{2\pi i \ell m/M}x(m)\right)_{m\in \mathbb{Z}_M}.
\end{equation*}
\noindent That is, $M_{\ell}$ multiplies $x = x(\cdot)$ pointwise with the harmonic $e^{2\pi i\ell(\cdot) /M}$.

\item The superposition $\pi(k,\ell) = M_{\ell}T_k$ of translation by $k$ and modulation by $\ell$ is a \emph{time-frequency shift operator}.

\item For $g\in\mathbb{C}^M\setminus \{0\}$ and $\Lambda\subset  \mathbb{Z}_M \times \mathbb{Z}_M$, the set of vectors 
\begin{equation*}
(g, \Lambda) = \{\pi (k, \ell)g\}_{(k,\ell)\in \Lambda}
\end{equation*}
\noindent is called the \emph{Gabor system} generated by the \emph{window} $g$ and the set $\Lambda$. A Gabor system  which spans $\mathbb{C}^M$ is a frame and is referred to as a \emph{Gabor frame}.
\end{enumerate}
\end{definition}

Here and in the sequel, we view a vector $x\in\mathbb{C}^M$ as a function $x:\mathbb{Z}_M\to~\mathbb{C}$, that is, all the operations on indices are done modulo $M$ and $x(m-k) = x(M+m -k)$. A more detailed description of Gabor frames in finite dimensions and their properties can be found in \cite{pfander2}.

\medskip

In this paper, we investigate extreme singular values of the analysis matrices of Gabor frames with random windows. As columns of such matrices are time-frequency shifts of the window vector, their entries, rows and columns are not independent. We see that, unlike the case of random matrices with independent entries, the singular values of a random time-frequency structured matrix depend not only on its dimensions, but also on the structure of the frame set~$\Lambda$. One of our aims therefore is to study this dependence by obtaining bounds on the singular values of analysis matrices of Gabor frames with frame sets having different structure. We also compare the obtained results to the respective results for random frames with independent entries. We show that, in the case of a generic Gabor frame, that is, for a randomly selected~$\Lambda$, singular values of the analysis matrix are close to the singular values of a random matrix with independent entries that has same dimensions.

The remaining part of this paper is organizes as follows. In Section \ref{sec_overwiew_propert} we give a brief overview of the results on the singular values of tall-and-skinny random matrices with independent identically distributed entries. We formulate the main results of this paper and compare them to the analogous results on matrices with independent entries in Section \ref{sec_main_results}. These result are then proven in Section~\ref{sec_proof_of_main_results}.  In Section~\ref{sing_val_sec},  we analyze the case when Gabor frame set has a particular simple structure $\Lambda = F\times\mathbb{Z}_M$ (or~$\mathbb{Z}_M\times F$), $F\subset\mathbb{Z}_M$, and show that the analysis matrix in this case is well-conditioned if the Gabor window $g$ is not to ``spiky''.  Finally, Section \ref{sec_num_res_sing_val}  contains numerical analysis of the singular values of random time-frequency structured matrices and discussion of the direction for further research. Appendix contains the probabilistic tools and results used in this paper.

\subsection{Related work}\label{sec_overwiew_propert}

Before we study the case when $\Phi$ is a Gabor frame, we include here a short overwiew of the results on the singular values for random frames with independent entries, such as Gaussian matrices. The largest singular value of the analysis matrix $\Phi^*$ of a random frame with independent entries can be estimated using Latala’s theorem \cite{latala2005some}. The following result implies that, with high probability,~${\sigma_{\max}(\Phi^*) = O\left( \sqrt{\frac{N}{M}}\right)}$.

\begin{theorem}\emph{\bf \cite{latala2005some}}\label{th_latala}
Let $\Phi^*\in \mathbb{C}^{N\times M}$, $N>M$, be a random matrix whose entries $\varphi_j(m)$, ${j\in \{1,\dots, N\}}$, $m\in \mathbb{Z}_M$, are independent identically distributed centered random variables, normalized so that ${\Var(\varphi_j(m)) = \frac{1}{M}}$. Assume further that $\mathbb{E}\left( |\varphi_j(m)|^4 \right)\le \frac{B}{M^2}$ for some constant $B>1$. Then there exists a constant $C>0$ depending only on $B$, such that
\begin{equation*}
\mathbb{E} \left( \sigma_{\max}(\Phi^*) \right)\le C\sqrt{\frac{N}{M}}.
\end{equation*}
\end{theorem}

The following optimal estimate of the smallest singular value of the analysis matrix for a random subgaussian frame with independent entries is due to Rudelson and Vershynin \cite{rudelson2009smallest}.

\begin{theorem}\emph{\bf \cite{rudelson2009smallest}}\label{th_rudelson_vershynin}
Let $\Phi^*\in \mathbb{C}^{N\times M}$, $N>M$, be a random matrix with entries $\varphi_j(m)$, $j\in \{1,\dots, N\}$, $m\in \mathbb{Z}_M$, that are independent identically distributed $L$-subgaussian random variables with zero mean, normalized so that ${\Var(\varphi_j(m)) = \frac{1}{M}}$. Then, for any $\varepsilon\ge 0$,
\begin{equation*}
\mathbb{P} \left\lbrace \sigma_{\min}(\Phi^*) >  \varepsilon \left(\sqrt{\frac{N}{M}} - \sqrt{\frac{M-1}{M}}\right)\right\rbrace \geq 1 - (C\varepsilon)^{N-M + 1}  + c^N,
\end{equation*}
\noindent where constants $C > 0$ and $c \in (0, 1)$ depend only on $L$.
\end{theorem}

We note that the estimate given by this result is tight also for square matrices, that is, when~${N = M}$.

\medskip

To the best of our knowledge, singular values of the analysis matrix $\Phi_\Lambda^*$ of a Gabor frame with random window and general $\Lambda$ with $|\Lambda|>M$ were not studied before. At the same time, the following bounds on the singular values of the synthesis matrix $\Phi$ of a Gabor system $(g,\Lambda)$ with a Steinhaus window $g$ and $|\Lambda| < M$ have been established in \cite{pfander2010sparsity}.

\begin{theorem}\emph{\bf \cite{pfander2010sparsity}}
Let $g$ be a Steinhaus window, that is, $g(j) = \frac{1}{\sqrt{M}}e^{2\pi i y_j}$, $j\in \mathbb{Z}_M$, with $y_j$ independent uniformly distributed on~$[0,1)$. Consider a Gabor system $(g, \Lambda)$ and let $\varepsilon, \delta \in (0,1)$. Suppose further that
\begin{equation*}
|\Lambda| \le \frac{\delta^2 M}{4e (\log (|\Lambda|/\varepsilon) + c)},
\end{equation*}
\noindent where $c = \log(e^2 / (4(e - 1)))\approx 0.0724$. Then $||I_\lambda - \Phi_\Lambda^* \Phi_\Lambda||_2\le \delta$ with probability at least $1 - \varepsilon$.

In other words the minimal and maximal singular values of $\Phi_\Lambda$ satisfy $$1 - \delta\le  \sigma_{\min}^2(\Phi_\Lambda) \le  \sigma_{\max}^2(\Phi_\Lambda) \le  1+\delta$$ with probability at least $1 - \varepsilon$.
\end{theorem}

\section{Main results}\label{sec_main_results}

In this paper we study Gabor frames with random windows and their analysis matrices. One of the main distributions for the window vectors considered in this paper is Steinhaus distribution, that is defined as follows.

\begin{definition}
A random vector $g\in \mathbb{S}^{M-1}$, such that $g(m) = \frac{1}{\sqrt{M}}e^{2\pi i y_m}$, $m\in \mathbb{Z}_M$ with $y_m$ independent uniformly distributed on $[0,1)$, is called a \emph{Steinhaus vector}.
\end{definition}

As we mentioned before, singular values of the analysis matrix $\Phi_\Lambda^*$ of a Gabor frame depend on the structure of the frame set $\Lambda$. In this paper we obtain the following bound on the largest singular value $\sigma_{\max}^2(\Phi_\Lambda^*)$ that holds for all $\Lambda$, independently of its structure, and only depends of the cardinality of $\Lambda$. One should consider this result as the worst case bound, since, as we see in Section \ref{sing_val_sec}, Example~\ref{example_structured}, much better bounds can be established for sets $\Lambda$ with specific structure.

\begin{theorem}\label{th_m2_max_sing_val_Gabor}
Let $g\in \mathbb{C}^M$ be a Steinhaus window and consider a Gabor system $(g, \Lambda)$ with $\Lambda\subset \mathbb{Z}_M\times \mathbb{Z}_M$. Then, for each fixed $\varepsilon\in (0,1)$, with probability at least $1 - \varepsilon$,
\begin{equation*}
\sigma_{\max}^2(\Phi_\Lambda^*)\le \frac{|\Lambda|}{M} + \sqrt{\frac{|\Lambda|}{\varepsilon}\left(1 - \frac{|\Lambda|}{M^2}\right)}.
\end{equation*}
\end{theorem}

We note that, the bound obtained in Theorem \ref{th_m2_max_sing_val_Gabor} is tight for a full Gabor frame, \linebreak when~${\Lambda = \mathbb{Z}_M\times \mathbb{Z}_M}$. In the case when $|\Lambda| = \alpha M^2$, for some $\alpha\in (0,1)$, the proven bound gives $\sigma_{\max}^2(\Phi_\Lambda^*)\le \left(\alpha + \sqrt{\frac{\alpha(1 - \alpha)}{\varepsilon}}\right)M = \left(1 + \sqrt{\frac{(1 - \alpha)}{\alpha\varepsilon}}\right)\frac{|\Lambda|}{M}$ with probability at least $1 - \varepsilon$. That is, the bound on $\sigma_{\max}^2(\Phi_\Lambda^*)$ in this case is the same (up to a constant), as the one obtained in Theorem \ref{th_latala} for matrices with independent identically distributed entries with bounded fourth moment.

\medskip

In this paper, we also obtain bounds on the extreme singular values of the analysis matrix of a Gabor frame with a randomly selected frame set $\Lambda$.  Roughly speaking, the obtained result shows that, for any $\epsilon\in (0,1)$, a randomly selected subframe $(g, \Lambda)$ of the full Gabor frame $(g, \mathbb{Z}_M\times \mathbb{Z}_M)$ with $|\Lambda| = O(M^{1+\epsilon}\log M)$ is well-conditioned with high probability.

\begin{theorem}\label{th_sing_val_rand_lambda}
Let $g\in \mathbb{C}^M$ be a Steinhaus window. For any fixed even $m\in \mathbb{N}$, consider a Gabor system $(g, \Lambda)$ with a random set $\Lambda\subset \mathbb{Z}_M\times \mathbb{Z}_M$ constructed so that events $\{(k,\ell)\in \Lambda\}$ are independent for all $(k,\ell)\in \mathbb{Z}_M\times \mathbb{Z}_M$ and have probability $\tau = \frac{C\log M}{M^{\frac{m-1}{m}}}$, where $C>0$ is a sufficiently large constant depending only on $m$. Then, with high probability (with respect to the choice of $\Lambda$),
\begin{equation*}
\mathbb{P}\left\lbrace\frac{|\Lambda|}{M}(1-\delta) \le \sigma_{\min}^2(\Phi_\Lambda^*) \le \sigma_{\max}^2(\Phi_\Lambda^*)\le \frac{|\Lambda|}{M}(1+\delta)\right\rbrace\ge 1 - \varepsilon,
\end{equation*}
\noindent where $\varepsilon\in (0,1)$ depends on $m$, $\delta$, and the choice of $C$.
\end{theorem}

We note that these bounds show the same asymptotic behavior as bounds on the extreme singular values of matrices with independent entries obtained in Theorem \ref{th_latala} and Theorem \ref{th_rudelson_vershynin}. This observation suggests that, for most of the choices of the frame set $\Lambda$, random time-frequency structured matrices are nearly as well-conditioned, as random matrices with independent Gaussian entries.

\section{Gabor analysis matrices with structured frame set $\Lambda$}\label{sing_val_sec}

Before we discuss the dependence of the optimal frame bounds on the structure and cardinality of $\Lambda$ and prove Theorems \ref{th_m2_max_sing_val_Gabor} and \ref{th_sing_val_rand_lambda}, let us consider a simple case when set $\Lambda$ has a particular structure. Namely, we start with the following observation.

\begin{proposition}\label{prop_sing_val_Gabor_regular}
Let $(g, \Lambda)$ be a Gabor system with $\Lambda = F\times \mathbb{Z}_M$  for some $F\subset \mathbb{Z}_M$, $F\ne \emptyset$, and a window $g\in \mathbb{C}^M$. Then $(g, \Lambda)$ is a frame if and only if $\min_{m\in \mathbb{Z}_M}\{||g_{F_m}||_2\}\ne 0$, where $g_{F_m}$ is the restriction of the vector $g$ to the set of coefficients $F_m = \{m - k\}_{k\in F}\subset \mathbb{Z}_M$. 

Moreover, in this case the optimal lower and upper frame bounds for $(g, \Lambda)$ are \linebreak$A=M\min_{m\in \mathbb{Z}_M}\{||g_{F_m}||_2^2\}$ and $B=M\max_{m\in \mathbb{Z}_M}\{||g_{F_m}||_2^2\}$, respectively.
\end{proposition}

\begin{proof}
Consider the matrix $\Phi_\Lambda\in \mathbb{C}^{M\times |F| M}$ corresponding to the synthesis operator of the Gabor system $(g, \Lambda)$, where $\Lambda = F\times \mathbb{Z}_M$ with $F\subset \mathbb{Z}_M$, $F\ne \emptyset$, and $g\in \mathbb{C}^M$. That is, the vectors $\pi(\lambda)g$, $\lambda\in \Lambda$, are the columns of the matrix $\Phi_\Lambda$. Then  consider the matrix $\Phi_\Lambda\Phi_\Lambda^*$ corresponding to the frame operator of $(g, \Lambda)$. For any $m_1, m_2\in \mathbb{Z}_M$,
\begin{equation*}
\begin{split}
\Phi_\Lambda\Phi_\Lambda^*(m_1, m_2) & = \sum_{\lambda\in \Lambda}(\pi(\lambda)g)(m_1)\overline{(\pi(\lambda)g)(m_2)} \\
& = \sum_{k\in F} \sum_{\ell\in \mathbb{Z}_M} e^{2\pi i \ell (m_1 - m_2)/M}g(m_1 - k)\overline{g(m_2 - k)}\\
& = \sum_{k\in F} g(m_1 - k)\overline{g(m_2 - k)} \sum_{\ell\in \mathbb{Z}_M} e^{2\pi i \ell (m_1 - m_2)/M}.
\end{split}
\end{equation*}
\noindent Then, since $\sum_{\ell\in \mathbb{Z}_M} e^{2\pi i \ell (m_1 - m_2)/M} = 0$ for $m_1\ne m_2$, and ${\sum_{\ell\in \mathbb{Z}_M} e^{2\pi i \ell (m_1 - m_2)/M} = M}$ for ${m_1 = m_2}$, we obtain
\begin{equation*}
\Phi_\Lambda\Phi_\Lambda^*(m_1, m_2) = \left\lbrace \begin{array}{ll}
0, & m_1\ne m_2\\
M \sum_{k\in F} |g(m_1 - k)|^2, & m_1 = m_2. 
\end{array}\right.
\end{equation*}
That is, $\Phi_\Lambda\Phi_\Lambda^* = \diag\{M \sum_{k\in F} |g(m - k)|^2\}_{m\in \mathbb{Z}_M}$ is a diagonal matrix and, thus, the set $\{\sigma_m(\Phi_\Lambda^*)\}_{m\in \mathbb{Z}_M}$ of the singular values of the matrix $\Phi_\Lambda^*$, corresponding to the \linebreak analysis operator of $(g, \Lambda)$, is equal to the set $\{\sqrt{M} ||g_{F_m}||_2\}_{m\in \mathbb{Z}_M}$, where \linebreak$F_m = \{m - k\}_{k\in F}\subset \mathbb{Z}_M$ and $g_S$ denotes the restriction of the vector $g$ to a set of coefficients~$S\subset \mathbb{Z}_M$.

In particular, $(g, \Lambda)$ is a frame if and only if all the diagonal entries of $\Phi_\Lambda\Phi_\Lambda^*$ are nonzero, that is, if and only if $\min_{m\in \mathbb{Z}_M}\{||g_{F_m}||_2\}\ne 0$. Moreover, we have
\begin{equation*}
\begin{split}
\sigma_{\min}(\Phi_\Lambda^*) & = \min_{m\in \mathbb{Z}_M}\sigma_m(\Phi_\Lambda^*) = \sqrt{M}\min_{m\in \mathbb{Z}_M}\{||g_{F_m}||_2\};\\
\sigma_{\max}(\Phi_\Lambda^*) & = \max_{m\in \mathbb{Z}_M}\sigma_m(\Phi_\Lambda^*) = \sqrt{M}\max_{m\in \mathbb{Z}_M}\{||g_{F_m}||_2\}.
\end{split}
\end{equation*}
That is, $M\min_{m\in \mathbb{Z}_M}\{||g_{F_m}||_2^2\}$ and $M\max_{m\in \mathbb{Z}_M}\{||g_{F_m}||_2^2\}$ are the optimal lower and upper frame bounds for $(g, \Lambda)$, respectively.
\end{proof}

\begin{remark} We note that an analogous result is true for the the case when the considered  set $\Lambda$ is of the form $\Lambda=\mathbb{Z}_M\times F$, for some $F\subset \mathbb{Z}_M$. Indeed, let $W_M = \frac{1}{\sqrt{M}}\{e^{-2\pi i k\ell/M}\}_{k,\ell\in \mathbb{Z}_M}$ be the normalized Fourier matrix, and consider the Gabor frame $(g, \Lambda')$ with a window $g$ and $\Lambda' = (-F)\times \mathbb{Z}_M$. Since $W_M M_{\ell} T_k g = e^{2\pi i k\ell /M} M_{-k} T_{\ell} W_M g$, we have 
\begin{equation*}
\begin{split}
W_M \Phi_{(g, \Lambda')} \Phi_{(g, \Lambda')}^* W_M^* (m_1,m_2) & = \sum_{(k,\ell)\in \Lambda} M_{-k}T_{\ell}W_M g(m_1)\overline{M_{-k}T_{\ell}W_M g(m_2)} \\
& =  \Phi_{(W_M g, \Lambda)} \Phi_{(W_M g, \Lambda)}^*(m_1,m_2).
\end{split}
\end{equation*}
That is, $W_M \Phi_{(g, \Lambda')} \Phi_{(g, \Lambda')}^* W_M^* = \Phi_{(W_M g, \Lambda)} \Phi_{(W_M g, \Lambda)}^*$. Thus, 
\begin{equation*}
\begin{gathered}
\sigma_{\min}(\Phi_{(W_Mg, \Lambda)}^*) = \sigma_{\min}(\Phi_{(g, \Lambda')}^*),\\
\sigma_{\max}(\Phi_{(W_Mg, \Lambda)}^*) = \sigma_{\max}(\Phi_{(g, \Lambda')}^*).
\end{gathered}
\end{equation*}

\end{remark}

\medskip

Let us now consider several particular classes of random Gabor windows and use Proposition~\ref{prop_sing_val_Gabor_regular} to estimate the frame bounds for the respective Gabor frames with the frame set of the form ${\Lambda = F\times \mathbb{Z}_M}$.

\begin{example}\label{example_structured}\mbox{}
\begin{enumerate}
\item[(i)] {\it Steinhaus window.} We first consider the case when the window $g$ is chosen so that ${g(m) = \frac{1}{\sqrt{M}}e^{2\pi i y_m}}$, $m\in \mathbb{Z}_M$, and $y_m$ are independent uniformly distributed on~$[0,1)$. Then, for each $m\in \mathbb{Z}_M$, $M \sum_{k\in F} |g(m - k)|^2 = |F|$, and thus $\Phi_\Lambda\Phi_\Lambda^* = |F|I_M$. That is, $(g, \Lambda)$ is a tight frame in this case.

\item[(ii)] {\it Gaussian window.} For a Gaussian window $g\sim \mathcal{C}\mathcal{N}\left( 0, \frac{1}{M}I_M\right)$, we have 
\begin{equation*}
\sigma_m^2(\Phi_\Lambda^*) = M \sum_{k\in F} |g(m - k)|^2 = \sum_{k\in F} \left(\frac{1}{2}2M r(m - k)^2 + \frac{1}{2}2M s(m - k)^2\right),
\end{equation*}
\noindent where $r(m-k) = \Re(g(m - k))$ denotes the real part of $g(m - k)$, and $s(m-k) = \Im(g(m - k))$ denotes its imaginary part. Since, for $k\in F$, $\sqrt{2M}r(m-k),~ \sqrt{2M}s(m-k) \sim \text{i.i.d. } \mathcal{N}(0,1)$ are independent standard Gaussian random variables, we can apply Lemma \ref{chi_square} to obtain that, for any $t>0$,
\begin{equation*}
\begin{split}
& \mathbb{P}\left\lbrace \sigma_m^2(\Phi_\Lambda^*) \ge |F| + \sqrt{2|F|t} + t \right\rbrace \le e^{-t};\\
& \mathbb{P}\left\lbrace \sigma_m^2(\Phi_\Lambda^*) \le |F| - \sqrt{2|F|t} \right\rbrace \le e^{-t}.
\end{split}
\end{equation*}
\noindent Then, setting $t =2 |F|$ in the first equation and $t = \frac{1}{8}|F|$ in the second one, we obtain
\begin{equation*}
\begin{split}
& \mathbb{P}\left\lbrace \sigma_m^2(\Phi_\Lambda^*) \ge 5|F| \right\rbrace \le e^{-2|F|};\\
& \mathbb{P}\left\lbrace \sigma_m^2(\Phi_\Lambda^*) \le \frac{1}{2}|F| \right\rbrace \le e^{-\frac{|F|}{8}}.
\end{split}
\end{equation*}
Suppose now that $|F|\ge C\log M$, for some sufficiently large constant $C>0$. Then, combining the probability estimates obtained above and taking the union bound over all $m\in \mathbb{Z}_M$, we obtain that, with high probability,
\begin{equation*}
\frac{1}{2}|F|< \sigma_m^2(\Phi_\Lambda^*)< 5|F|,
\end{equation*}
\noindent for all $m\in \mathbb{Z}_M$. In particular, for the frame bounds of $(g, \Lambda)$ we have 
\begin{equation}\label{example_gaussian_wind}
\frac{1}{2}|F|< \sigma_{\min}^2(\Phi_\Lambda^*) \le \sigma_{\max}^2(\Phi_\Lambda^*)< 5|F|.
\end{equation}

\item[(iii)] {\it Window, uniformly distributed on $\mathbb{S}^{M-1}$.} It is a well-known fact that a window $g$, uniformly distributed on the unit sphere $\mathbb{S}^{M-1}$, can be written in the form $g = h/||h||_2$, where $h\sim \mathcal{C}\mathcal{N}\left(0,\frac{1}{M} I_M\right)$ \cite{marsaglia1972choosing}. Moreover, Lemma \ref{lemma_norm_gaussian} shows that, for some $C>0$, $\frac{1}{2}\le ||h||_2\le 2$ with probability at least $1 - e^{-CM}$. Thus, with the same probability, 
\begin{equation*}
\frac{1}{4} M \sum_{k\in F} |h(m - k)|^2 \le M \sum_{k\in F} |g(m - k)|^2 \le  4M \sum_{k\in F} |h(m - k)|^2.
\end{equation*}
\noindent Combining this with \eqref{example_gaussian_wind}, we obtain that with high probability
\begin{equation*}
\frac{1}{8}|F|< \sigma_{\min}^2(\Phi_\Lambda^*) \le \sigma_{\max}^2(\Phi_\Lambda^*)< 20|F|.
\end{equation*}
\end{enumerate}
\end{example}

\medskip
The examples above show that, in the case when $\Lambda$ has a regular structure and window $g$ is random, the Gabor frame $(g,\Lambda)$ has frame bounds that are quite close to each other, and, thus, is well-conditioned.  

\section{Gabor analysis matrices with general frame set $\Lambda$}\label{sec_proof_of_main_results}

In this section we are going to consider the case when $\Lambda$ is a generic subset of $\mathbb{Z}_M\times\mathbb{Z}_M$ and prove Theorems \ref{th_m2_max_sing_val_Gabor} and \ref{th_sing_val_rand_lambda}. We start our consideration by showing the following technical lemma, which follows the idea of \cite[Lemma~3.4]{pfander2010sparsity}. Here and in the sequel, we denote the identity $M\times M$ matrix by $I_M$.

\begin{lemma}\label{lemma_trace_formula}
Let $g$ be a Steinhaus window, that is, $g(j) = \frac{1}{\sqrt{M}}e^{2\pi i y_j}$, $j\in \mathbb{Z}_M$, with $y_j$ independent uniformly distributed on~$[0,1)$. Consider a Gabor system $(g, \Lambda)$ with $\Lambda\subset \mathbb{Z}_M\times \mathbb{Z}_M$. Then, for any $m\in \mathbb{N}$ and $\delta >0$,
\begin{equation*}
\mathbb{P}\left\lbrace\frac{|\Lambda|}{M}(1-\delta) \le \sigma_{\min}^2(\Phi_\Lambda^*) \le \sigma_{\max}^2(\Phi_\Lambda^*)\le \frac{|\Lambda|}{M}(1+\delta)\right\rbrace \geq 1 - \frac{M^{2m}}{|\Lambda|^{2m}}\delta^{-2m}\mathbb{E}(\Tr H^{2m}),
\end{equation*}
\noindent where $H = \Phi_\Lambda \Phi_\Lambda ^*- \frac{|\Lambda|}{M}I_M $. Furthermore, for any $m\in \mathbb{N}$,
\begin{equation*}
\mathbb{E}\left(\Tr H^m \right) = \sum_{\substack{j_1, j_2, \dots, j_m\in \mathbb{Z}_M,\\j_1\ne j_2\ne\dots\ne j_m\ne j_{1}}}\sum_{(k_1, \ell_1)\in \Lambda}  \dots \sum_{(k_m, \ell_m)\in \Lambda} e^{\frac{2\pi i}{M} \sum_{t = 1}^m \ell_t (j_t - j_{t+1})} E_{\substack{j_1\dots j_m\\k_1\dots k_m}},
\end{equation*}
\noindent where $E_{\substack{j_1\dots j_m\\k_1\dots k_m}} =\frac{1}{M^{m}}$, if there exists a bijection $\alpha:\{1,\dots,m\}\to\{1,\dots,m\}$, such that ${j_t - k_t = j_{\alpha(t)} - k_{\alpha(t) - 1}}$, for all $t\in\{1,\dots,m\}$; and $E_{\substack{j_1\dots j_m\\k_1\dots k_m}} = 0$, otherwise.
\end{lemma}

\begin{proof}
First, for $H = \Phi_\Lambda \Phi_\Lambda ^*-  \frac{|\Lambda|}{M}I_M$, we note that 
\begin{equation*}
\mathbb{P}\left\lbrace\frac{|\Lambda|}{M}(1-\delta) \le \sigma_{\min}^2(\Phi_\Lambda^*) \le \sigma_{\max}^2(\Phi_\Lambda^*)\le \frac{|\Lambda|}{M}(1+\delta)\right\rbrace  = \mathbb{P}\left\lbrace ||H||_2 \le \frac{|\Lambda|}{M}\delta\right\rbrace.
\end{equation*}

Using Markov’s inequality, the fact that the Frobenius norm majorizes the operator norm, and the fact that $H$ is self-adjoint, for any $m\in \mathbb{N}$ we have
\begin{align*}
\mathbb{P} \left\lbrace||H||_2> \frac{|\Lambda|}{M}\delta\right\rbrace & = \mathbb{P} \left\lbrace||H||_2^{2m}> \frac{|\Lambda|^{2m}}{M^{2m}}\delta^{2m}\right\rbrace \le \frac{M^{2m}}{|\Lambda|^{2m}}\delta^{-2m}\mathbb{E}(||H||_2^{2m})  \\
& = \frac{M^{2m}}{|\Lambda|^{2m}}\delta^{-2m}\mathbb{E}(||H^m||_2^{2}) \le \frac{M^{2m}}{|\Lambda|^{2m}}\delta^{-2m}\mathbb{E}(||H^m||_F^{2}) \\
& = \frac{M^{2m}}{|\Lambda|^{2m}}\delta^{-2m}\mathbb{E}(\Tr H^{2m}).\numberthis \label{eq_trace_estimate}
\end{align*}
\noindent That is, to conclude the desired result, we aim to estimate the trace expectation $\mathbb{E}(\Tr H^{2m})$. For any $j_1, j_2\in \mathbb{Z}_M$, 
\begin{equation*}
\Phi_\Lambda \Phi_\Lambda^* (j_1, j_2) = \sum_{(k, \ell)\in \Lambda}e^{2\pi i \ell (j_1 - j_2)/M}g(j_1 - k)\overline{g(j_2 - k)}.
\end{equation*}
Thus, since $|g(j)| = \frac{1}{\sqrt{M}}$, for all $j\in \mathbb{Z}_M$, for $H$ we have
\begin{equation*}
H (j_1, j_2) = \left\lbrace \begin{array}{ll}
\sum_{(k, \ell)\in \Lambda}e^{2\pi i \ell (j_1 - j_2)/M}g(j_1 - k)\overline{g(j_2 - k)}, & j_1\ne j_2;\\
0, & j_1 = j_2.
\end{array}\right.
\end{equation*}
Then, for $j_1,\dots, j_{m+1}\in \mathbb{Z}_M$, we recursively obtain
\begin{equation*}
\Scale[0.93]{
\begin{split}
& H^2 (j_1, j_3) = \sum_{j_2\in \mathbb{Z}_M} H(j_1,j_2)H(j_2, j_3) \\
& = \sum_{\substack{j_2\in \mathbb{Z}_M,\\j_2\ne j_1,j_3}}\sum_{(k_1, \ell_1)\in \Lambda} \sum_{(k_2, \ell_2)\in \Lambda} e^{\frac{2\pi i}{M} (\ell_1 (j_1 - j_2) + \ell_2 (j_2 - j_3))} g(j_1 - k_1)\overline{g(j_2 - k_1)}g(j_2 - k_2)\overline{g(j_3 - k_2)};\\
& H^3 (j_1, j_4) = \sum_{j_3\in \mathbb{Z}_M} H^2(j_1,j_3)H(j_3, j_4) \\
& = \sum_{\substack{j_3\in \mathbb{Z}_M,\\j_3\ne j_4}}\sum_{\substack{j_2\in \mathbb{Z}_M,\\j_2\ne j_1,j_3}}\sum_{(k_1, \ell_1)\in \Lambda} \sum_{(k_2, \ell_2)\in \Lambda} \sum_{(k_3, \ell_3)\in \Lambda} e^{\frac{2\pi i}{M} \sum_{t = 1}^3 \ell_t (j_t - j_{t+1})} \prod_{t = 1}^3 g(j_t - k_t)\overline{g(j_{t+1} - k_t)};
\end{split}}
\end{equation*}
and, in general,
\begin{equation*}
\Scale[0.93]{
\begin{split}
& H^m (j_1, j_{m+1}) = \sum_{j_m\in \mathbb{Z}_M} H^{m-1}(j_1,j_m)H(j_m, j_{m+1}) \\
& = \sum_{\substack{j_m\in \mathbb{Z}_M,\\j_m\ne j_{m+1}}}\dots\sum_{\substack{j_3\in \mathbb{Z}_M,\\j_3\ne j_4}}\sum_{\substack{j_2\in \mathbb{Z}_M,\\j_2\ne j_1,j_3}}\sum_{(k_1, \ell_1)\in \Lambda}  \dots \sum_{(k_m, \ell_m)\in \Lambda} e^{\frac{2\pi i}{M} \sum_{t = 1}^m \ell_t (j_t - j_{t+1})} \prod_{t = 1}^m g(j_t - k_t)\overline{g(j_{t+1} - k_t)}.
\end{split}}
\end{equation*}
Thus, for the trace of the matrix $H^m$, we have
\begin{equation*}
\Scale[0.93]{
\begin{split}
& \Tr(H^m) = \sum_{\substack{j_1, j_2, \dots, j_m\in \mathbb{Z}_M,\\j_1\ne j_2\ne\dots\ne j_m\ne j_{1}}}\sum_{(k_1, \ell_1)\in \Lambda}  \dots \sum_{(k_m, \ell_m)\in \Lambda} e^{\frac{2\pi i}{M} \sum_{t = 1}^m \ell_t (j_t - j_{t+1})} \prod_{t = 1}^m g(j_t - k_t)\overline{g(j_{t+1} - k_t)}, \text{ and}\\
& \mathbb{E}\left(\Tr(H^m)\right) = \sum_{\substack{j_1, j_2, \dots, j_m\in \mathbb{Z}_M,\\j_1\ne j_2\ne\dots\ne j_m\ne j_{1}}}\sum_{(k_1, \ell_1)\in \Lambda}  \dots \sum_{(k_m, \ell_m)\in \Lambda} e^{\frac{2\pi i}{M} \sum_{t = 1}^m \ell_t (j_t - j_{t+1})} E_{\substack{j_1\dots j_m\\k_1\dots k_m}},
\end{split}}
\end{equation*}
\noindent where $ E_{\substack{j_1\dots j_m\\k_1\dots k_m}} = \mathbb{E}\left(\prod_{t = 1}^m g(j_t - k_t)\overline{g(j_{t+1} - k_t)}\right)$.

Let us compute $E_{\substack{j_1\dots j_m\\k_1\dots k_m}}$ now. Since $g(j)$, $j\in \mathbb{Z}_M$, are independent, the expectation can be factored into a product of the form
\begin{equation*}
\mathbb{E}\left(\prod_{t = 1}^m g(j_t - k_t)\overline{g(j_{t+1} - k_t)}\right) = \prod_{j\in \mathbb{Z}_M}\mathbb{E}\left( g(j)^{\mu_j}\overline{g(j)}^{\nu_j} \right),
\end{equation*}
\noindent for some $\mu_j, \nu_j\in \mathbb{N}\cup \{0\}$. Moreover, since $\sqrt{M}g(j)$ is uniformly distributed on the unit torus $\{z\in \mathbb{C}: ||z||_2 = 1\}$ and $\mathbb{E}\left( g(j) \right) = 0$, we have
\begin{equation*}
\mathbb{E}\left( g(j)^{\mu_j}\overline{g(j)}^{\nu_j} \right) = \left\lbrace \begin{array}{ll}
\mathbb{E}\left( |g(j)|^{2\mu_j} \right) = \frac{1}{M^{\mu_j}}, & \mu_j = \nu_j;\\
0, & \mu_j \ne \nu_j.
\end{array}\right.
\end{equation*}
Thus, under the convention that $k_{0} = k_{m}$,
\begin{equation*}
E_{\substack{j_1\dots j_m\\k_1\dots k_m}} = \left\lbrace \begin{array}{ll}
\frac{1}{M^{m}}, & \text{if} ~ \exists \text{ bijection } \alpha: \{1,\dots,m\} \to \{1,\dots,m\}, \\
&\text{s.t. }\forall t\in \{1,\dots,m\} ~  j_t - k_t = j_{\alpha(t)} - k_{\alpha(t) - 1}; \\
0, & \text{otherwise.}
\end{array}\right.
\end{equation*}
This concludes the proof.
\end{proof}

\subsection{Proof of Theorem \ref{th_m2_max_sing_val_Gabor}}

To prove Theorem \ref{th_m2_max_sing_val_Gabor}, we apply Lemma \ref{lemma_trace_formula} with $m = 1$. We obtain that, for $H = \Phi_\Lambda \Phi_\Lambda ^* - \frac{|\Lambda|}{M}I_M $ and any $\delta >0$,
\begin{equation*}
\mathbb{P}\left\lbrace\sigma_{\max}^2(\Phi_\Lambda^*)> \frac{|\Lambda|}{M}(1+\delta)\right\rbrace \le \frac{M^{2}}{|\Lambda|^{2}}\delta^{-2}\mathbb{E}(\Tr H^{2}).
\end{equation*}
\noindent And, moreover, we have
\begin{equation*}
\mathbb{E}\left(\Tr H^2 \right) = \sum_{\substack{j_1, j_2\in \mathbb{Z}_M,\\j_1\ne j_2}}\sum_{(k_1, \ell_1)\in \Lambda} \sum_{(k_2, \ell_2)\in \Lambda} e^{\frac{2\pi i}{M} (\ell_1 - \ell_2) (j_1 - j_2)} E_{\substack{j_1 j_2\\k_1 k_2}},
\end{equation*}
\noindent where $E_{\substack{j_1 j_2\\k_1 k_2}} =\frac{1}{M^{2}}$, if there exists a bijection $\alpha: \{1,2\} \to \{1,2\}$, such that, for every $t\in \{1,2\}$, $j_t - k_t = j_{\alpha(t)} - k_{\alpha(t) - 1}$; and $E_{\substack{j_1 j_2\\k_1 k_2}} = 0$, otherwise. If we have $j_1 - k_1 = j_{2} - k_{1}$, or $j_2 - k_2 = j_{1} - k_{2}$, it follows that $j_1 = j_2$, which is a contradiction. Thus we have
\begin{equation*}
\begin{array}{lcl}
\left\lbrace \begin{array}{l}
j_1 - k_1 = j_{1} - k_{2}\\
j_2 - k_2 = j_{2} - k_{1}
\end{array}\right. & \Leftrightarrow & k_1 = k_2.
\end{array}
\end{equation*}
For each $k\in \mathbb{Z}_M$, let us consider the set $A_k = \{\ell \in \mathbb{Z}_M, (k, \ell)\in \Lambda\}$. Clearly, $\sum_{k\in \mathbb{Z}_M}|A_k| = |\Lambda|$. Then, for the expectation of $\Tr H^2$, we have the following.
\begin{equation*}
\begin{split}
\mathbb{E}\left(\Tr H^2 \right) & = \frac{1}{M^2}\sum_{\substack{j_1, j_2\in \mathbb{Z}_M,\\j_1\ne j_2}}\sum_{k\in \mathbb{Z}_M}\sum_{\ell_1\in A_k} \sum_{\ell_2 \in A_k} e^{\frac{2\pi i}{M} (\ell_1 - \ell_2) (j_1 - j_2)} \\
& = \frac{1}{M^2}\sum_{j_1\in \mathbb{Z}_M}\sum_{k\in \mathbb{Z}_M}\sum_{\ell_1\in A_k}\left(\sum_{\substack{\ell_2 \in A_k\\\ell_2\ne \ell_1}}\sum_{\substack{j_2\in \mathbb{Z}_M,\\j_2\ne j_1}} e^{\frac{2\pi i}{M} (\ell_1 - \ell_2) (j_1 - j_2)} + \sum_{\substack{j_2\in \mathbb{Z}_M,\\j_2\ne j_1}} 1\right) \\
& = \frac{1}{M^2}\sum_{j_1\in \mathbb{Z}_M}\sum_{k\in \mathbb{Z}_M}\sum_{\ell_1\in A_k}\left(\sum_{\substack{\ell_2 \in A_k\\\ell_2\ne \ell_1}}(-1) + M - 1 \right)\\
& =  \frac{1}{M^2}M\sum_{k\in \mathbb{Z}_M} \sum_{\ell_1\in A_k} \left((|A_k| - 1)(-1) + M - 1 \right) = \frac{1}{M}\sum_{k\in \mathbb{Z}_M} |A_k| \left(M - |A_k| \right)\\
& = \sum_{k\in \mathbb{Z}_M} |A_k|  - \frac{1}{M} \sum_{k\in \mathbb{Z}_M} |A_k|^2 \le |\Lambda| \left(1 - \frac{|\Lambda|}{M^2} \right).
\end{split}
\end{equation*}
The last step here is due to the fact that $\sum_{k\in \mathbb{Z}_M} |A_k|^2 \ge \frac{1}{M} \left( \sum_{k\in \mathbb{Z}_M} |A_k| \right)^2$.

Then, setting $\delta = \sqrt{\frac{M^2 - |\Lambda|}{\varepsilon |\Lambda|}}$ for some $\varepsilon\in (0,1)$, we obtain
\begin{equation*}
\mathbb{P}\left\lbrace\sigma_{\max}^2(\Phi_\Lambda^*)> \frac{|\Lambda|}{M} + \sqrt{\frac{|\Lambda|}{\varepsilon}\left(1 - \frac{|\Lambda|}{M^2}\right)} \right\rbrace \le \varepsilon,
\end{equation*}
which concludes the proof on Theorem \ref{th_m2_max_sing_val_Gabor}.

\subsection{Proof of Theorem \ref{th_sing_val_rand_lambda}}

It follows from Lemma \ref{lemma_trace_formula}, that, for matrix $H = \Phi_\Lambda \Phi_\Lambda ^* -  \frac{|\Lambda|}{M}I_M$, every even~$m\in \mathbb{N}$, and $\delta >0$,
\begin{equation*}
\begin{gathered}
\mathbb{P}\left\lbrace\frac{|\Lambda|}{M}(1-\delta) \le \sigma_{\min}^2(\Phi_\Lambda^*) \le \sigma_{\max}^2(\Phi_\Lambda^*)\le \frac{|\Lambda|}{M}(1+\delta)\right\rbrace \geq 1 - \frac{M^{m}}{|\Lambda|^{m}}\delta^{-m}\mathbb{E}(\Tr H^{m}),\\
\mathbb{E}\left(\Tr H^m \right) = \sum_{\substack{j_1, j_2, \dots, j_m\in \mathbb{Z}_M,\\j_1\ne j_2\ne\dots\ne j_m\ne j_{1}}}\sum_{(k_1, \ell_1)\in \Lambda}  \dots \sum_{(k_m, \ell_m)\in \Lambda} e^{\frac{2\pi i}{M} \sum_{t = 1}^m \ell_t (j_t - j_{t+1})} E_{\substack{j_1\dots j_m\\k_1\dots k_m}},
\end{gathered}
\end{equation*}
\noindent where $E_{\substack{j_1\dots j_m\\k_1\dots k_m}} =\frac{1}{M^{m}}$ if there exists a permutation $\alpha\in \Sigma_m$, such that, for  every $t\in \{1,\dots,m\}$, $j_t - k_t = j_{\alpha(t)} - k_{\alpha(t) - 1}$; and $E_{\substack{j_1\dots j_m\\k_1\dots k_m}} = 0$ otherwise. Here, $\Sigma_m$ denotes the group of permutations of $\{1,\dots, m\}$. That is, $\alpha\in \Sigma_m$ is a bijection $\alpha:\{1,\dots , m\} \to \{1,\dots , m\}$.

For $k\in \mathbb{Z}_M$, let us denote by $A_k$ the set $A_k = \{\ell\in \mathbb{Z}_M, \text{ s.t. }(k,\ell)\in \Lambda\}$. After rearranging the sum in the trace formula above, we have
\begin{equation*}
\begin{split}
\mathbb{E}\left(\Tr H^m \right) & = \sum_{\substack{j_1, j_2, \dots, j_m\in \mathbb{Z}_M,\\j_1\ne j_2\ne\dots\ne j_m\ne j_{1}}}\sum_{k_1, k_2, \dots, k_m\in \mathbb{Z}_M} E_{\substack{j_1\dots j_m\\k_1\dots k_m}} \sum_{\ell_1\in A_{k_1}}  \dots \sum_{\ell_m\in A_{k_m}} e^{\frac{2\pi i}{M} \sum_{t = 1}^m \ell_t (j_t - j_{t+1})}\\
& = \sum_{\substack{j_1, j_2, \dots, j_m\in \mathbb{Z}_M,\\j_1\ne j_2\ne\dots\ne j_m\ne j_{1}}}\sum_{k_1, k_2, \dots, k_m\in \mathbb{Z}_M} E_{\substack{j_1\dots j_m\\k_1\dots k_m}} \prod_{t = 1}^m \sum_{\ell_t\in A_{k_t}} e^{\frac{2\pi i}{M}\ell_t (j_t - j_{t+1})}
\end{split}
\end{equation*}
We note that, by the construction of $\Lambda$, each set $A_{k_t}$, $t\in \{1,\dots, m\}$, is a random subset of $\mathbb{Z}_M$, such that the events $\{\ell\in A_{k_t}\}$, $\ell\in \mathbb{Z}_M$, are independent and have probability $\tau$. Then Corollary \ref{cor_sum_rand_roots_of_unity} implies that, for every $t\in \{1,\dots, m\}$ and a constant $C' > 4\sqrt{2}$,
\begin{equation*}
\mathbb{P}\left\lbrace \max_{q\in \mathbb{Z}_M, q\ne 0} \left|\sum_{\ell\in A_{k_t}} e^{2\pi i \ell q/M}\right| < C'\log M \right\rbrace \ge 1 -  \frac{1}{M^{\frac{C'}{2\sqrt{2}} - 2}}.
\end{equation*}
\noindent In particular, $$\max_{\substack{j_t, j_{t+1}\in \mathbb{Z}_M, \\ j_t\ne j_{t+1}}}\left|\sum_{\ell_t\in A_{k_t}} e^{\frac{2\pi i}{M}\ell_t (j_t - j_{t+1})}\right| < C'\log M,$$ with probability at least $1 -  \dfrac{1}{M^{\frac{C'}{2\sqrt{2}} - 2}}$. By taking the union bound over all $t\in \{1,\dots, m\}$, we conclude that, with probability at least ${1-\dfrac{m}{M^{\frac{C'}{2\sqrt{2}}-2}}}$,
\begin{equation*}
\left|\prod_{t = 1}^m \sum_{\ell_t\in A_{k_t}} e^{\frac{2\pi i}{M}\ell_t (j_t - j_{t+1})}\right| < C'^m \log^m M.
\end{equation*}
Then, applying the triangular inequality to the trace formula, we obtain that, on an event $X$ of probability at least $1 -  \dfrac{m}{M^{\frac{C'}{2\sqrt{2}} - 2}}$,
\begin{equation*}
\begin{split}
\mathbb{E}\left(\Tr H^m \right) & \le \sum_{\substack{j_1, j_2, \dots, j_m\in \mathbb{Z}_M,\\j_1\ne j_2\ne\dots\ne j_m\ne j_{1}}}\sum_{k_1, k_2, \dots, k_m\in \mathbb{Z}_M} E_{\substack{j_1\dots j_m\\k_1\dots k_m}} \left|\prod_{t = 1}^m \sum_{\ell_t\in A_{k_t}} e^{\frac{2\pi i}{M}\ell_t (j_t - j_{t+1})}\right| \\
& < C'^m \log^m M \sum_{\substack{j_1, j_2, \dots, j_m\in \mathbb{Z}_M,\\j_1\ne j_2\ne\dots\ne j_m\ne j_{1}}}\sum_{k_1, k_2, \dots, k_m\in \mathbb{Z}_M} E_{\substack{j_1\dots j_m\\k_1\dots k_m}}
\end{split}
\end{equation*}

A permutation $\alpha\in \Sigma_m$ can be presented as a product
\begin{equation}\label{eq_cycle_decomp}
\alpha = (i_{11} i_{12}\dots i_{1 r_1})(i_{21} i_{22}\dots i_{2 r_2})\dots (i_{s1} i_{s2}\dots i_{s r_s})
\end{equation}
\noindent of disjoint cycles, where $r_1 + r_2 +\dots + r_s = m$, and, for each $p\in \{1,\dots , s\}$, $\alpha(i_{p q}) = i_{p (q+1)}$ for $q\in \{1,\dots, r_p - 1\}$ and $\alpha(i_{p r_p }) = i_{p 1}$.

Suppose that we have $k_1,\dots , k_m$ fixed. Then $E_{\substack{j_1\dots j_m\\k_1\dots k_m}}\ne 0$ if and only if there exists $\alpha\in \Sigma_m$, such that $j_t -  j_{\alpha(t)} = k_t - k_{\alpha(t) - 1}$, for all $t\in \{1,\dots,m\}$. 
Assuming that $\alpha$ has $s$ cycles in the disjoint cycle decomposition \eqref{eq_cycle_decomp}, this condition can be rewritten in the form of $s$ systems of linear equations for $j_1,\dots, j_m$. Namely, for each $p\in \{1,\dots , s\}$, we have 
\begin{align*}
j_{i_{p1}} -  j_{i_{p2}} & = k_{i_{p1}} - k_{i_{p2} - 1}\\
j_{i_{p2}} -  j_{i_{p3}} & = k_{i_{p2}} - k_{i_{p3} - 1}\\
& \cdots\\
j_{i_{pr_p}} -  j_{i_{p1}} & = k_{i_{pr_p}} - k_{i_{p1} - 1}.\numberthis\label{system_for_j}
\end{align*}  
Note that the system \eqref{system_for_j} has rank $r_p-1$. Furthermore, summing up all the equations, on the left hand side we obtain zero. So, \eqref{system_for_j} has $M$ different solutions~if
\begin{equation}\label{eq_restrict_k}
\sum_{q = 1}^{r_p}k_{i_{pq}} = \sum_{q = 1}^{r_p}k_{i_{pq} - 1},
\end{equation}
and does not have a solution otherwise. Moreover, if $s\ne 1$, that is, $r_p<m$, then the sets of indices $\{i_{pq}\}_{q = 1}^{r_p}$ on the left hand side of \eqref{eq_restrict_k} and $\{i_{pq} - 1\}_{q = 1}^{r_p}$ on the right hand side of \eqref{eq_restrict_k} are different. Indeed, suppose that $\{i_{pq}\}_{q = 1}^{r_p} = \{i_{pq} - 1\}_{q = 1}^{r_p}$, and let $i_{pq_0} = \min_{q \in \{1,\dots, r_p\}} i_{pq}$ be the smallest element in this set. Since $i_{pq_0} - 1$ is also an element of $\{i_{pq}\}_{q = 1}^{r_p}$, we have $i_{pq_0} - 1\ge i_{pq_0}$, which implies $i_{pq_0} = 1$ and $i_{pq_0} - 1= m$. Then, since $m\in \{i_{pq}\}_{q = 1}^{r_p}$, we also have  $m -1\in \{i_{pq}\}_{q = 1}^{r_p}$. Proceeding the argument by induction, we obtain $\{i_{pq}\}_{q = 1}^{r_p} = \{1, \dots, m\}$, which is a contradiction. Without loss of generality, we can assume that $i_{pr_p}\notin \{i_{pq} - 1\}_{q = 1}^{r_p}$, for every $p\in \{1,\dots,s\}$. 

It follows that, for each cycle in the cycle decomposition~\eqref{eq_cycle_decomp}, except the last one, equation~\eqref{eq_restrict_k} is a nontrivial linear relation for $k_t$, $t \in \{1,\dots, m\}$. For the last cycle the relation follows automatically, assuming \eqref{eq_restrict_k} is satisfied for each $p\in \{1,\dots, s-1\}$. So, for the system of linear equations for $j_1,\dots, j_m$ to have a solution, $k_{i_{pr_p}}$, $p\in \{1,\dots , s-1\}$, should be determined by $\{k_1,\dots, k_m\}\setminus \{k_{i_{pr_p}}\}_{p = 1}^{ s-1}$ using equations \eqref{eq_restrict_k}. It this case the number of different solutions is~$M^s$.

Then, for the expectation of the trace of $H^m$, on the event $X$ we have
\begin{equation*}
\begin{split}
\mathbb{E}\left(\Tr H^m \right) & < C'^m \log^m M \sum_{\substack{j_1, j_2, \dots, j_m\in \mathbb{Z}_M,\\j_1\ne j_2\ne\dots\ne j_m\ne j_{1}}}\sum_{k_1, k_2, \dots, k_m\in \mathbb{Z}_M} E_{\substack{j_1\dots j_m\\k_1\dots k_m}}\\
& \le C'^m \log^m M \sum_{s = 1}^m S(m,s)  \sum_{j_{i_{11}},\dots, j_{i_{s1}}\in \mathbb{Z}_M} \sum_{\substack{k_{i_{11}},\dots, k_{i_{1(r_1-1)}}\in \mathbb{Z}_M \\ ^{\vdots} \\ k_{i_{(s-1)1}},\dots, k_{i_{(s-1)(r_s-1)}}\in \mathbb{Z}_M \\ k_{i_{s1}},\dots, k_{i_{sr_s}}\in \mathbb{Z}_M}} \frac{1}{M^m}\\
& = C'^m\frac{\log^m M}{M^m} \sum_{s = 1}^m S(m,s) M^s M^{m-s+1} \\
& = C'^m M\log^m M \sum_{s = 1}^m S(m,s) = C'^m m! M\log^m M,
\end{split}
\end{equation*}
\noindent where $S(m, s)$ denotes the Stirling number of the first kind, equal to the number of permutations in $\Sigma_m$ with exactly $s$ cycles in the disjoint cycle decomposition.

Moreover, the cardinality of $\Lambda$ is given by a sum of $M^2$ independent Bernoulli random variables with success probability $\tau = \frac{C\log M}{M^{\frac{m-1}{m}}}$. More precisely, $$|\Lambda| = \sum_{(k,\ell)\in \mathbb{Z}_M\times \mathbb{Z}_M}{\bf 1}_{\Lambda}(k,\ell).$$ Then Hoeffding's inequality (Lemma \ref{Hoeffging_ineq_Bernoulli}) applied with $t = \frac{C\log M}{2M^{\frac{m-1}{m}}}$ implies
\begin{equation*}
\mathbb{P}\left\lbrace |\Lambda|\le \frac{1}{2}CM^{1+ \frac{1}{m}}\log M \right\rbrace\le e^{-2C^2M^{\frac{2}{m}}\log^2 M}.
\end{equation*}
\noindent That is, $|\Lambda| > \frac{1}{2}CM^{1+ \frac{1}{m}}\log M$ on an event $Y$ of probability at least $1 - e^{-2C^2M^{\frac{2}{m}}\log^2 M}$.

Then, on the event $X\cap Y$, which has probability at least $1 - \frac{\tilde{C}m}{M^{\frac{C'}{2\sqrt{2}} - 2}}$, for some $\tilde{C}>0$, the obtained estimates for the trace expectation and frame set cardinality lead to the following probability bound for the singular values estimates.
\begin{equation*}
\begin{split}
\mathbb{P}\left\lbrace\frac{|\Lambda|}{M}(1-\delta) \le \sigma_{\min}^2(\Phi_\Lambda^*) \le \sigma_{\max}^2(\Phi_\Lambda^*)\le \frac{|\Lambda|}{M}(1+\delta)\right\rbrace \geq 1 - \frac{M^{m}}{|\Lambda|^{m}}\delta^{-m}\mathbb{E}(\Tr H^{m})\\
\geq 1 - C'^m m! \delta^{-m} \frac{M^{m}}{\frac{1}{2^m}C^m M^{m+1}\log^m M}M\log^m M = 1 - \left(\frac{2C'}{C}\right)^m m! \delta^{-m}.
\end{split}
\end{equation*}
This concludes the proof of Theorem \ref{th_sing_val_rand_lambda} provided $C$ is chosen to be large enough.

\section{Numerical results and further discussion}\label{sec_num_res_sing_val}

In this section we further investigate singular values of Gabor frames with a random window using numerical simulations. In particular, we aim to numerically analyse the bounds on the extreme singular values of the analysis matrix $\Phi_\Lambda^*$ of a Gabor frame $(g,\Lambda)$ with a Steinhaus window $g$ in the case when $\Lambda$ is a random subset of $\mathbb{Z}_M\times\mathbb{Z}_M$.

Let us fix an even $m\in\mathbb{N}$, and let $C>0$ be a sufficiently large constant depending on $m$. Consider a random $\Lambda\subset \mathbb{Z}_M\times\mathbb{Z}_M$, such that the events $\{(k,\ell)\in \Lambda\}$ are independent for all $(k,\ell)\in \mathbb{Z}_M\times \mathbb{Z}_M$ and have probability ${\tau = \frac{C\log M}{M^{\frac{m-1}{m}}}}$. Theorem~\ref{th_sing_val_rand_lambda} ensures that, with high probability (with respect to the choice of $\Lambda$),
\begin{equation*}
\mathbb{P}\left\lbrace\frac{|\Lambda|}{M}(1-\delta) \le \sigma_{\min}^2(\Phi_\Lambda^*) \le \sigma_{\max}^2(\Phi_\Lambda^*)\le \frac{|\Lambda|}{M}(1+\delta)\right\rbrace\ge 1 - \varepsilon,
\end{equation*}
\noindent where $\varepsilon\in (0,1)$ depends on $m$, $\delta$, and the choice of $C$. To illustrate Theorem \ref{th_sing_val_rand_lambda}, we use two sets of numerical simulations. 

\begin{figure}[t]\center
\begin{tabular}{cc}
\includegraphics[width=74mm]{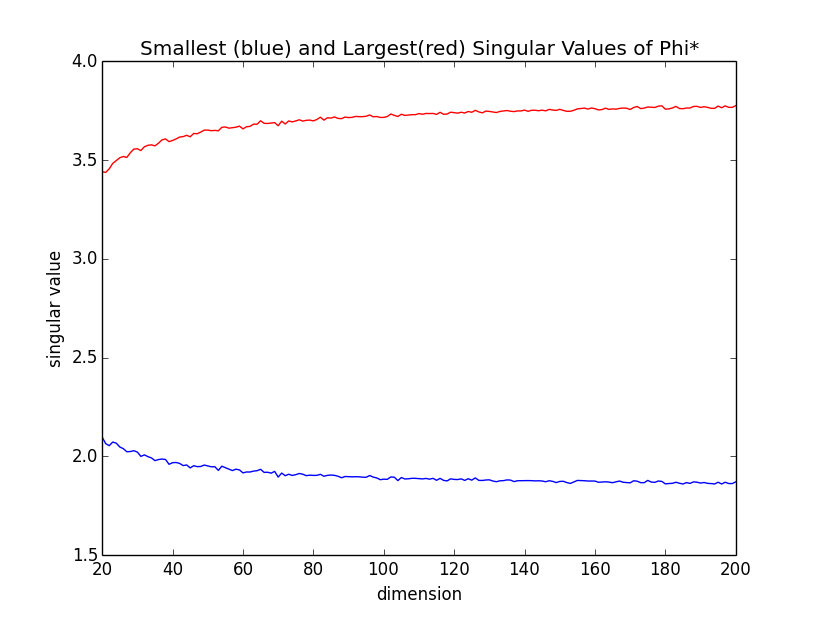}
&
\includegraphics[width=74mm]{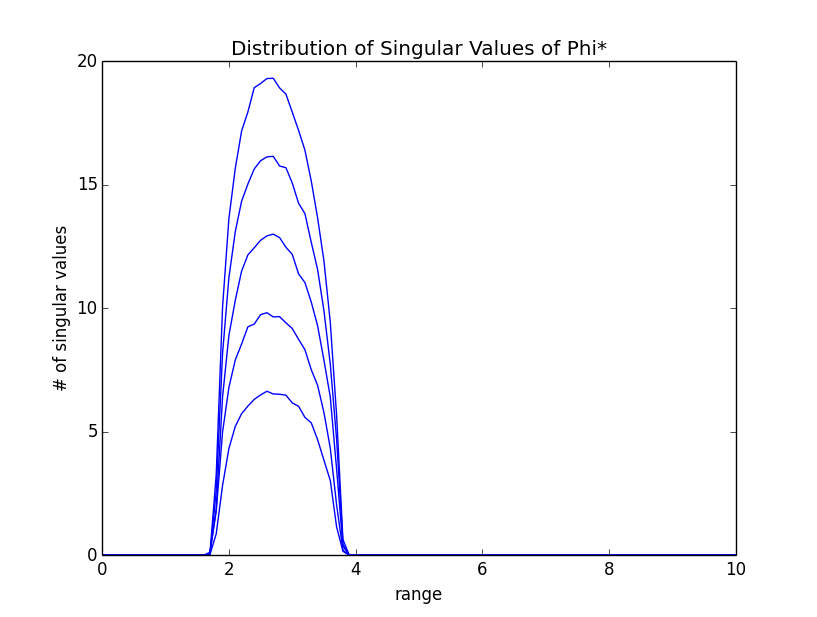}
\end{tabular}
\caption{\label{num_sing_val_distr} The left hand side of the figure shows the dependence of the extreme singular values of the analysis matrix $\Phi_{\Lambda}^*$ of a Gabor frame $(g, \Lambda)$ on the ambient dimension $M$; and the right hand side of the figure shows the distribution of the singular values of $\Phi_{\Lambda}^*$ for the dimensions $M = 100, 150, 200, 250, 300$. Here, $g$ is a Steinhaus window and $\Lambda$ is chosen at random as described in Theorem \ref{th_sing_val_rand_lambda}, with $\tau = \frac{C}{M}$, that is, $|\Lambda| = O(M)$ with high probability. The number of the numerical experiments considered here is 1000. These numerical results suggest that, with high probability, the singular values of $\Phi_{\Lambda}^*$ lie inside an interval $\left[k\frac{|\Lambda|}{M}, K\frac{|\Lambda|}{M}\right]$, for some constants $0<k<K$ that do not depend on $M$. This allows us to conjecture that a version of Theorem~\ref{th_sing_val_rand_lambda} is true also for~$\Lambda$ with $|\Lambda| = O(M)$. In other words, the additional factor of $M^\epsilon \log M$ in the cardinality of $\Lambda$ is a side effect of the method used to prove the theorem.}
\end{figure}

In the first set of numerical simulations, we investigate the behavior of the singular values of the analysis matrix $\Phi_{\Lambda}^*$ of a Gabor frame $(g, \Lambda)$ with a Steinhaus window $g$ and set $\Lambda\subset \mathbb{Z}_M\times\mathbb{Z}_M$ selected at random, so that $|\Lambda| = O(M)$ with high probability. The obtained numerical results suggest that, in the case when random $\Lambda$ is constructed as described in Theorem \ref{th_sing_val_rand_lambda} with $\tau = \frac{C}{M}$, there exist constants $0<k<K$ not depending on the ambient dimension $M$, such that all the singular values of the analysis matrix $\Phi_{\Lambda}^*$ are inside the interval $\left[k\frac{|\Lambda|}{M}, K\frac{|\Lambda|}{M}\right]$ with high probability, see Figure \ref{num_sing_val_distr} (left). The right hand side of Figure \ref{num_sing_val_distr} shows the distribution of the singular values of $\Phi_{\Lambda}^*$ over this interval for the selected dimensions $M = 100, 150, 200, 250, 300$. 
\medskip

\begin{figure}[t]\center
\begin{tabular}{cc}
\includegraphics[width=76mm]{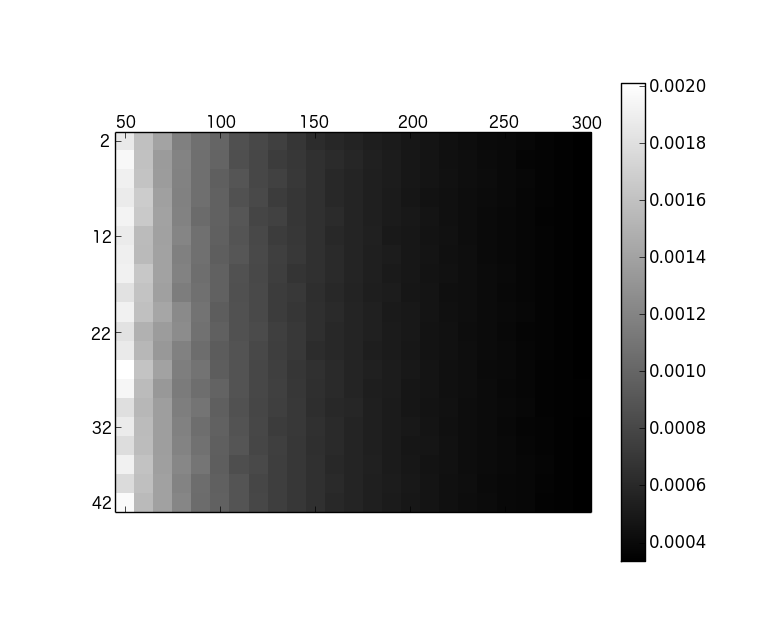}
&
\includegraphics[width=76mm]{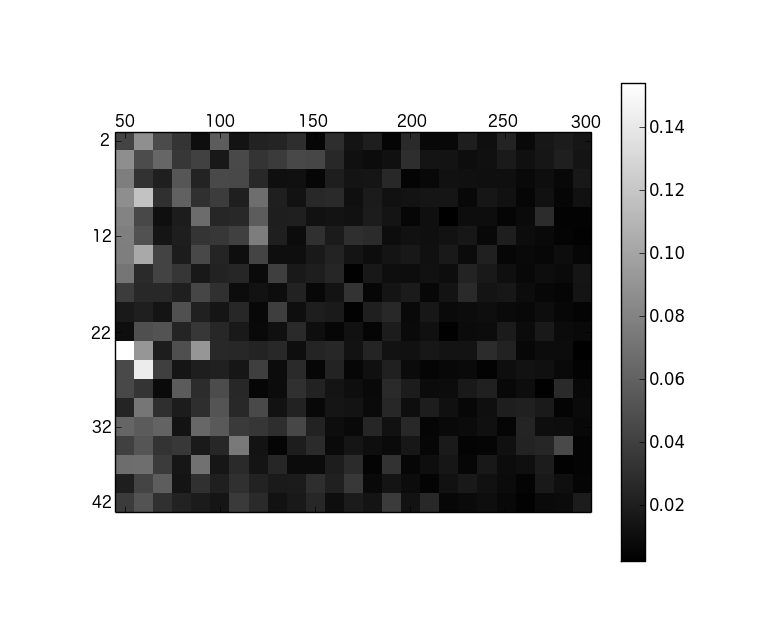}
\end{tabular}
\caption{\label{num_sing_val_trace_est} The figure illustrates the behavior of the numerically estimated normalized trace expectation $\frac{M^{m}}{|\Lambda|^{m}}\mathbb{E}\left(\Tr \left(\Phi_\Lambda \Phi_\Lambda ^*- \frac{|\Lambda|}{M}I_M \right)^{m}\right)$, where $\Phi_\Lambda$ is the synthesis matrix of a Gabor frame $(g, \Lambda)$ with a Steinhaus window $g$. The left hand side of the figure illustrates the numerical results in the case when $\Lambda$ is chosen at random, as described in Theorem~\ref{th_sing_val_rand_lambda}, with $\tau = \frac{C}{M}$; and the right hand side of the figure illustrates the case when ${\Lambda = F\times \{0, 1, \dots, \lfloor\frac{M}{2}\rfloor\}}$. The plots show the dependence of the normalized trace expectation on the ambient dimension $M$ (horizontal axis) and the parameter $C$ (vertical axis), for a fixed $m$. These numerical results allow us to conjecture that the probability bound obtained in Theorem~\ref{th_sing_val_rand_lambda} can be further improved.
}
\end{figure}

We use the second set of simulations to investigate the behavior of the trace of the matrix $H = \Phi_\Lambda \Phi_\Lambda ^*- \frac{|\Lambda|}{M}I_M$, where $\Phi_\Lambda$ is the synthesis matrix of a Gabor frame $(g, \Lambda)$ with a Steinhaus window $g$. It follows from Lemma \ref{lemma_trace_formula} that $${\mathbb{P}\left\lbrace \sigma_{\min}^2(\Phi_\Lambda^*) \le \frac{|\Lambda|}{M}(1-\delta)  \text{ or } \sigma_{\max}^2(\Phi_\Lambda^*)\ge \frac{|\Lambda|}{M}(1+\delta)\right\rbrace \leq \frac{M^{2m}}{|\Lambda|^{2m}}\delta^{-2m}\mathbb{E}(\Tr H^{2m})}.$$ In other words, the normalized trace expectation $\frac{M^{m}}{|\Lambda|^{m}}\mathbb{E}\left(\Tr H^{m}\right)$ is used to estimate the probability of the ``failure'' event when either the minimal singular value of the analysis matrix $\Phi_\Lambda^*$ is too small or the maximal singular value is too large, so that $(g, \Lambda)$ is not well-conditioned.

For the normalized trace expectation, we consider two different constructions of~$\Lambda$, providing the average and the worst case estimates, respectively. The left hand side of Figure~\ref{num_sing_val_trace_est} shows the numerical results in the case when $\Lambda$ is chosen at random as described in Theorem \ref{th_sing_val_rand_lambda} with $\tau = \frac{C}{M}$. The right hand side of Figure~\ref{num_sing_val_trace_est} illustrates the case when $\Lambda$ is of the form $\Lambda = F\times \{0, 1, \dots, \lfloor\frac{M}{2}\rfloor\}$, $F\subset\mathbb{Z}_M$. The plots show the dependence of the normalized trace expectation on the ambient dimension $M$ and the parameter $C$ in the definition of $\tau$, for a fixed~$m$. The obtained numerical results suggest that, in both cases, the normalized trace expectation decreases rapidly with the dimension. This allows us to conjecture that the probability bound obtained in Theorem~\ref{th_sing_val_rand_lambda} can be further improved. Moreover, Figure \ref{num_sing_val_trace_est} (left) shows that, in the case of randomly selected $\Lambda$, the normalized trace expectation does not seem to depend on the parameter $C$.

These numerical findings, illustrated on Figures \ref{num_sing_val_distr} and \ref{num_sing_val_trace_est}, motivate the following conjecture.

\begin{conjecture}\label{conj_sing_val_Gabor_small_lambda}
Let $g$ be a Steinhaus window, that is, $g(j) = \frac{1}{\sqrt{M}}e^{2\pi i y_j}$, ${j\in \mathbb{Z}_M}$, with $y_j$ independent uniformly distributed on~$[0,1)$. Consider a Gabor system $(g, \Lambda)$ with a random set $\Lambda\subset \mathbb{Z}_M\times \mathbb{Z}_M$ constructed so that events $\{(k,\ell)\in \Lambda\}$ are independent for all $(k,\ell)\in \mathbb{Z}_M\times \mathbb{Z}_M$ and have probability $\tau = \frac{C}{M}$, where $C>0$ is a sufficiently large numerical constant. Then, with high probability (with respect to the choice of $\Lambda$),
\begin{equation*}
\mathbb{P}\left\lbrace\frac{|\Lambda|}{M}(1-\delta) \le \sigma_{\min}^2(\Phi_\Lambda^*) \le \sigma_{\max}^2(\Phi_\Lambda^*)\le \frac{|\Lambda|}{M}(1+\delta)\right\rbrace\ge 1 - \frac{c}{M},
\end{equation*}
\noindent where the constant $c>0$ depends only on $\delta$.
\end{conjecture}

In other words, the additional factor of $M^{\frac{1}{m}} \log M$ in the cardinality of $\Lambda$ is a side effect of the method used to prove Theorem \ref{th_sing_val_rand_lambda}. Analogous conjectures can be formulated also for other distributions of the Gabor window.

\subsection{Erasure-robust frames}

In some areas of signal processing related to communication systems or phase retrieval, the available frame coefficients of an (unknown) signal of interest are not only corrupted by additive noise, but some of them might be missing or be too unreliable to be used for reconstruction. In this case, the measurement frame should allow for robust signal reconstruction from incomplete set of noisy frame coefficients. Fickus and Mixon introduced the notion of numerically erasure-robust frames that formalizes this property \cite{fickus2012numerically}.

\begin{definition}
For a fixed $p\in [0,1]$ and $C \geq 1$, a frame $\Phi = \{\varphi_j\}_{j = 1}^N$ is called a \emph{$(p,C)$-numerically erasure-robust frame} if, for every $J \subset \{1,\dots,N\}$ of cardinality $|J| = (1 - p)N$, the condition number of the analysis matrix of the corresponding subframe $\Phi_J = \{\varphi_j\}_{j \in J}$ satisfies $\Cond(\Phi_J^*) \leq C$.
\end{definition}

The following result shows that a Gaussian frame with independent frame vectors is a numerically erasure-robust frame \cite{fickus2012numerically}.

\begin{theorem}
Fix $\varepsilon >0$ and consider a frame $\Phi = \{\varphi_j\}_{j = 1}^N\subset \mathbb{C}^M$ such that \linebreak$\varphi_j(m) \sim~\text{i.i.d. } \mathcal{C}\mathcal{N}(0,1)$, for $j\in \{1,\dots,N\}$ and $m\in \mathbb{Z}_M$. Then $\Phi$ is a $(p,C)$-numerically erasure-robust frame with overwhelming probability provided $p$ and $C$ satisfy
\begin{equation*}
 \sqrt{\frac{M}{N}}\leq \frac{C-1}{C+1}\sqrt{1 - p} - \sqrt{\varepsilon + 2p(1 - \log(p))}.
\end{equation*}
\end{theorem}

\medskip
In this section, we aim to numerically investigate robustness to erasures of Gabor frames~$(g, \Lambda)$ with a random window $g$.

We note that, since a full Gabor frame is tight, for any $\Lambda\subset \mathbb{Z}_M\times\mathbb{Z}_M$ and $g\in \mathbb{S}^{M-1}$,
\begin{equation*}
\sigma_{\max}^2(\Phi_\Lambda^*) = \max_{x\in \mathbb{S}^{M-1}}\sum_{\lambda\in \Lambda}|\langle x, \pi(\lambda)g\rangle|^2\le \max_{x\in \mathbb{S}^{M-1}}\sum_{\lambda\in \mathbb{Z}_M\times\mathbb{Z}_M}|\langle x, \pi(\lambda)g\rangle|^2 = M.
\end{equation*}
Thus we concentrate on the uniform bound on the minimal singular value $\sigma_{\min}^2(\Phi_{\Lambda'}^*) $, for all subframes $(g, \Lambda')$ of $(g, \Lambda)$ with $|\Lambda'| \ge (1 - p)|\Lambda|$, where $p$ is some fixed parameter.

\begin{figure}[t]\center
\includegraphics[width=90mm]{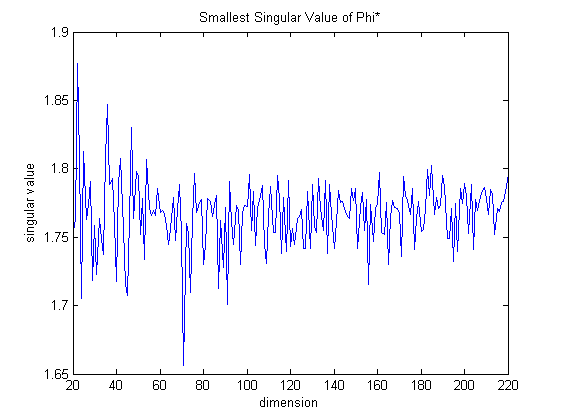}
\caption{\label{num_singular} The figure shows the dependence of the numerically estimated parameter ${\Delta\left(\frac{1}{3}\right)=\min\left\lbrace\sigma_{\min}^2(\Phi_{\Lambda'}^*):~ \Lambda'\subset \Lambda,~ |\Lambda'|\ge \frac{2}{3}|\Lambda|\right\rbrace}$ on the ambient dimension $M$. The Gabor window $g$ here is random, uniformly distributed on the unit sphere $\mathbb{S}^{M-1}$, and ${\Lambda = F\times\mathbb{Z}_M}$, where $|F|$ is a constant that does not depend on $M$. The number of the numerical experiments considered here is 1000, and the plot shows the smallest obtained result for each dimension. These numerical results suggest that $\Delta\left(\frac{1}{3}\right)$ is bounded away from zero by a numerical constant not depending on the dimension~$M$, that is, the Gabor frame $(g, F\times \mathbb{Z}_M)$ is robust to erasures.}
\end{figure}

For a Gabor frame $(g, \Lambda)$ with a window $g$ uniformly distributed on the unit sphere~$\mathbb{S}^{M-1}$, let us consider the parameter $\Delta(p)$, $p\in [0,1]$, given by
\begin{equation*}
\Delta(p) = \min_{\substack{\Lambda'\subset \Lambda,\\ |\Lambda'|\ge (1 - p)|\Lambda|}}\sigma_{\min}^2(\Phi_{\Lambda'}^*).
\end{equation*}
Numerical results illustrating the dependence of the value $\Delta(1/3)$ on the dimension $M$ are presented on Figure \ref{num_singular}. They suggest that $\Delta(p)$ is bounded away from zero by a numerical constant not depending on $M$. More precisely, we formulate the following conjecture.

\begin{conjecture}\label{conj_sing_val_gabor_unif}
Consider a Gabor frame $(g, \Lambda)$ with $g$ uniformly distributed on~$\mathbb{S}^{M-1}$ and $\Lambda\subset \mathbb{Z}_M\times\mathbb{Z}_M$, such that $|\Lambda| = O(M\log^\alpha M)$ (where the parameter $\alpha\geq 0$ has to be specified). Then, for $p\in (0,1)$, $\Delta(p)\geq C$ with high probability, where $C>0$ depends only on $p$.
\end{conjecture}

%

\section*{Appendix: Probability theory tools}\label{probability_background}

In this appendix we collect the probabilistic tools and results used in the proofs of Theorems \ref{th_m2_max_sing_val_Gabor} and \ref{th_sing_val_rand_lambda}. We start by stating the Hoeffding's inequality in the special case of Bernoulli random variables.

\begin{lemma}[Hoeffding's inequality]\label{Hoeffging_ineq_Bernoulli}
Let $X_j$, $j\in \{1,\dots N\}$, be independent identically distributed Bernoulli random variables, such that $\mathbb{P}\{X_j=1\} = p$, for some $p\in (0,1)$, that is $X_j \sim \text{i.i.d. } B\left(1,p\right)$. Consider the random variable ${S = \sum_{j = 1}^N X_j}$. Then, for every $t>0$, we have
\begin{equation*}
\begin{gathered}
\mathbb{P}\{S< (p - t)N\}\le e^{-2t^2N} \quad \text{and} \quad \mathbb{P}\{S> (p + t)N\}\le e^{-2t^2N}.
\end{gathered}
\end{equation*}
\end{lemma}

The following lemma, proven in \cite{laurent2000adaptive}, is useful for obtaining bounds on the norms of random vectors.

\begin{lemma}\emph{\textbf{\cite{laurent2000adaptive}}}\label{chi_square}
Let $Y_1,\dots,Y_M \sim  i.i.d. ~\mathcal{N}(0,1)$ and fix $c = (c_1,\dots,c_M)$ with $c_k\ge 0$, $k\in\{1,\dots, M\}$. Then, for $Z = \sum_{k = 1}^M c_k(Y_k^2 - 1)$ the following inequalities hold for any $t>0$.
\begin{equation}\label{lower}
\mathbb{P}\{Z \ge 2||c||_2 \sqrt{t} + 2||c||_\infty t\} \le e^{-t};
\end{equation}
\begin{equation}\label{upper}
\mathbb{P}\{Z \le -2||c||_2 \sqrt{t}\} \le e^{-t}.
\end{equation}

\end{lemma}

Using Lemma \ref{chi_square}, we obtain the following bounds on the norm of a random Gaussian vector $h\sim \mathcal{C}\mathcal{N}\left( 0, \frac{1}{M} I_M \right)$.

\begin{lemma}\label{lemma_norm_gaussian}
Consider a random vector $h\in \mathbb{C}^M$, such that $h\sim \mathcal{C}\mathcal{N}\left( 0, \frac{1}{M} I_M \right)$. Then, there exists a constant $C>0$, such that
\begin{equation*}
\mathbb{P}\left\lbrace\frac{1}{2} < ||h||_2 < 2\right\rbrace \ge 1 - e^{-CM}.
\end{equation*}
\end{lemma}

\begin{proof}
First, we note that $$2M ||h||_2^2 = 2M \sum_{k = 1}^M (|a_k|^2 + |b_k|^2),$$ where $h(k) = a(k) + ib(k)$ and $a(k), b(k)\sim i.i.d. ~ \mathcal{N}(0,\frac{1}{2M})$. Then, for any \linebreak$k \in \{1,\dots,M\}$, $\sqrt{2M}a(k), \sqrt{2M}b(k)$ are independent standard Gaussian random variables. We apply inequality (\ref{lower}) from Lemma \ref{chi_square} with $c_k = 1$, $k \in \{1,\dots,M\}$, to obtain that, for any $t>0$,
\begin{equation*}
\mathbb{P}\{2M||h||_2^2 \ge \sqrt{8Mt} + 2t + 2M\}\le e^{-t}.
\end{equation*}
\noindent Taking $t= M/2$, we have
\begin{equation}\label{eq_upper_norm_gaussian}
\mathbb{P}\{||h||_2^2 > 4\} = \mathbb{P}\{2M||h||_2^2 > 8M\} \le \mathbb{P}\{2M||h||_2^2 \ge 5M\}\le e^{-M/2}.
\end{equation}
Similarly, by applying inequality (\ref{upper}) from Lemma \ref{chi_square} with $c_k = 1$, we get
\begin{equation*}
\mathbb{P}\left\lbrace||h||_2^2 \le -\sqrt{\frac{2t}{M}} + 1\right\rbrace\le e^{-t},
\end{equation*}
\noindent for every $t>0$. Taking $t= 9M/32$, we obtain
\begin{equation}\label{eq_lower_norm_gaussian}
\mathbb{P}\left\lbrace||h||_2^2 \le -\sqrt{\frac{2t}{M}} + 1\right\rbrace = \mathbb{P}\left\lbrace||h||_2^2 \le \frac{1}{4}\right\rbrace\le e^{-9M/32},
\end{equation}
Summarizing the bounds obtained in \eqref{eq_upper_norm_gaussian} and \eqref{eq_lower_norm_gaussian}, we conclude the desired claim.
\end{proof}

\subsection*{Fourier bias}

In additive combinatorics, the notion of Fourier bias is used to measure pseudorandomness of a set. Roughly speaking, it helps to distinguish between sets which are highly uniform and behave like random sets, and those which are highly non-uniform and behave like arithmetic progressions \cite{tao}. 

\begin{definition}
Take $C\subset \mathbb{Z}_M$ and let ${\bf1}_C$ be the characteristic function of $C$. Then the \emph{Fourier bias} of $C$ is given by
\begin{equation*}
||C||_u = \max_{m\in \mathbb{Z}_M\setminus \{0\}}{|(\mathcal{F} {\bf1}_C)(m)|}.
\end{equation*}
\end{definition}

The following lemma follows from Chernoff's inequality and can be found in \cite[Lemma~4.16]{tao}. Loosely speaking, it shows that, if $B$ is a random subset of $A\subset \mathbb{Z}_M$, then $||B||_u$ is tightly concentrated around $\frac{|B|}{|A|}||A||_u$. In other words, the Fourier bias of a random subset scales proportionally to its cardinality.

\begin{lemma}\label{lemma_fourier_bias_rand}
Consider an additive subset $A$ of $\mathbb{Z}_M$ with $M>4$, and fix $0<\tau\le 1$. Let $B$ be a random subset of $A$, such that ${\bf 1}_{B} (a)\sim \text{i.i.d.}~B(1,\tau)$, for $a\in A$, that is, events $\{a\in B\}$ are independent and have probability $\tau$. Then, for any $\lambda>0$ and $\sigma^2 = \frac{|A|}{M^2}\tau (1 - \tau)$, we have
\begin{equation*}
\mathbb{P}\left\lbrace | ||B||_u - \tau ||A||_u| \ge \lambda \sigma \right\rbrace \le 4M \max \left\lbrace e^{-\frac{\lambda^2}{8}}, ~e^{-\frac{\lambda \sigma}{2\sqrt{2}}} \right\rbrace.
\end{equation*}
\end{lemma}

As an easy consequence of Lemma \ref{lemma_fourier_bias_rand}, we obtain the following result that provides an efficient bound on the absolute value of the sum of randomly sampled roots of unity.

\begin{corollary}\label{cor_sum_rand_roots_of_unity}
Let $B$ be a random subset of $\mathbb{Z}_M$, such that ${\bf 1}_{B} (m)\sim \text{i.i.d.}~B(1,\tau)$, for $m\in \mathbb{Z}_M$ and $0<\tau < 1$. Then, for any constant $C>4\sqrt{2}$, we have
\begin{equation*}
\mathbb{P}\left\lbrace \max_{m\in \mathbb{Z}_M\setminus \{0\}} \left|\sum_{b\in B} e^{2\pi i bm/M}\right| < C\log M \right\rbrace \ge 1 -  \frac{1}{M^{\frac{C}{2\sqrt{2}} - 2}}.
\end{equation*}
\end{corollary}

\begin{proof}
Let us apply Lemma \ref{lemma_fourier_bias_rand} with $A = \mathbb{Z}_M$. Then, since $||\mathbb{Z}_M||_u = 0$ and $\sigma^2 = \frac{|A|}{M^2}\tau (1 - \tau) = \frac{\tau (1 - \tau)}{M}$, for any $\lambda>0$ we obtain
\begin{equation*}
\mathbb{P}\left\lbrace  ||B||_u \ge \lambda \sqrt{\frac{\tau (1 - \tau)}{M}} \right\rbrace \le 4M \max \left\lbrace e^{-\frac{\lambda^2}{8}}, ~e^{-\frac{\lambda \sqrt{\tau (1 - \tau)}}{2\sqrt{2M}}} \right\rbrace.
\end{equation*}
Then, by choosing $\lambda = \frac{C}{\sqrt{\tau(1 - \tau)}}\sqrt{M}\log M$ with a constant $C>4\sqrt{2}$, we ensure that
\begin{equation*}
\Scale[0.93]{
\begin{split}
4M \max \left\lbrace e^{-\frac{\lambda^2}{8}}, ~e^{-\frac{\lambda \sqrt{\tau (1 - \tau)}}{2\sqrt{2M}}} \right\rbrace = \max \left\lbrace e^{-\frac{C^2 M\log^2 M}{8\tau(1 - \tau)} + \log (4M)}, ~e^{-\frac{C\log M}{2\sqrt{2}} + \log (4M)}\right\rbrace = \frac{1}{M^{\frac{C}{2\sqrt{2}} - 2}}.
\end{split}}
\end{equation*}
Thus, we obtain that
\begin{equation*}
\mathbb{P}\left\lbrace  ||B||_u \ge C\log M \right\rbrace \le \frac{1}{M^{\frac{C}{2\sqrt{2}} - 2}},
\end{equation*}
\noindent and $||B||_u = \max_{m\in \mathbb{Z}_M\setminus \{0\}}{|(\mathcal{F} {\bf1}_B)(m)|} = \max_{m\in \mathbb{Z}_M\setminus \{0\}} \left|\sum_{b\in B} e^{2\pi i bm/M}\right|$, which concludes the proof.
\end{proof}

\let\Section\section 
\def\section*#1{\Section{#1}} 
\bibliographystyle{plain}
\bibliography{geometric_properties}

\end{document}